\documentclass{article}
\usepackage{amsrefs}
\usepackage{amssymb}
\usepackage{amsmath}
\usepackage{amsthm}
\usepackage{graphicx}

\newtheorem{theorem}{Theorem}
\newtheorem{proposition}{Proposition}
\newtheorem{lemma}{Lemma}
\newtheorem{definition}{Definition}

\DeclareMathOperator{\im}{im}
\DeclareMathOperator{\divi}{div}
\DeclareMathOperator{\rank}{rk}
\DeclareMathOperator{\gon}{gon}
\DeclareMathOperator{\Cliff}{Cliff}
\newtheorem*{subject}{2000 Mathematics Subject Classification}
\newtheorem*{keywords}{Keywords}
\newtheorem*{TheoremA}{Theorem A}
\newtheorem*{TheoremB}{Theorem B}
\newtheorem*{TheoremC}{Theorem C}
\theoremstyle{remark}
\newtheorem*{remark}{Remark}
\newtheorem*{claim}{Claim}

\author{Marc Coppens\footnote{KU  Leuven, Department of Mathematics, Section of Algebra,
Celestijnenlaan 200B bus 2400 B-3001 Leuven and department Elektrotechniek (ESAT), Belgium; email: marc.coppens@kuleuven.be.}}
\title{A study of general Martens-special chains of cycles. }
\date{}

\begin{document}
\maketitle \noindent

\begin{abstract}
For a general Martens-special chain of cycles $\Gamma$ of type $k$ we prove that the gonality is equal to $k+2$.
Although $\dim (W^1_{k+2} (\Gamma))=k$ we prove that $w^1_{k+2}(\Gamma)=0$.
We also compute the gonality sequence of $\Gamma$ and we prove it is divisorial complete.
We prove that a general Martens-special discrete chain of cycles $G$ of type $k$ has the same gonality sequence.
\end{abstract}

\begin{subject}
05C25, 14T15
\end{subject}

\begin{keywords}
Brill-Noether numbers, divisors, gonality, gonality sequence, divisorial complete, metric graphs, finite graphs, tableaux
\end{keywords}

\section{Introduction}\label{section1}

In this paper, studying general Martens-special chains of cycles, we illustrate some differences between the theories of divisors on metric graphs and divisors on smooth curves.

\begin{itemize}
\item On a metric graph $\Gamma$ the difference between $\dim (W^r_d(\Gamma))$ and the Brill-Noether number $w^r_d(\Gamma)$ can become very large.
\item On a metric graph $\Gamma$ large Brill-Noether numbers $w^r_d(\Gamma)$ give very poor information on the gonality of $\Gamma$ if $d-2r \neq 0$.
\item For any integer $c \geq 0$ there exist metric graphs $\Gamma$ of arbitrary large genus with Clifford index equal to $c$ that are divisorial complete. 
\end{itemize}

Let us first recall some facts relevant to this paper from divisor theory on smooth curves.

Let $C$ be a smooth curve of genus $g$ and let $J(C)$ be the Jacobian of $C$ parametrizing linear equivalence classes of divisors $D$ of degree 0 on $C$.
Choosing a base point $P_0 \in C$ one considers $W^r_d(C)=\{ D \in J(C) : \dim \vert D+dP_0 \vert \geq r \}$.
An element $D$ of $W^r_d(C) \setminus W^{r+1}_d(C)$ gives rise to a complete linear system $g^r_d$ equal to $\vert D+dP_0 \vert$.
We refer to \cite{ref4} for the theory of divisors on curves.

In \cite{ref3} the main result is the so-called H. Martens' Theorem.
If $r$ and $d$ are integers with $1 \leq r \leq g-2$ and $0 \leq d \leq g-2+r$ then $\dim (W^r_d(C)) \leq d-2r$ and $\dim (W^r_d (C))=d-2r$ if and only if $C$ is hyperelliptic.
The curve $C$ is hyperelliptic if and only if $W^1_2(C) \neq \emptyset$ (i.e. $C$ has a $g^1_2$).

Later on more results are proved giving rise to the statement that a ''large'' dimension of some $W^r_d(C)$ gives rise to a ''small'' gonality of $C$ (the gonality of $C$ is the smallest integer $e$ such that $C$ has a $g^1_e$).
See e.g. \cite{ref4} (Chapter IV, Theorem 5.2 and exercise series G of Chapter IV), \cite{a} and \cite{b} (Theorem 15).

Fixing a degree $d$, the Riemann-Roch Theorem gives a lower bound on $r$ such that $C$ has a complete $g^r_d$.
Also it immediately implies the existence of complete $g^r_d$ in case $d \leq g$ and $r=0$ and in case $g \leq d \leq 2g-2$ and $r=d-g+1$.
The Clifford index $\Cliff (g^r_d)$ of a complete linear system $g^r_d$ is defined to be $d-2r$.
The Clifford index $\Cliff (C)$ of a smooth curve $C$ is the minimal value $\Cliff (g^r_d)$ for complete linear systems satisfying $r \geq 1$ and $r \geq d-g+2$.
Fixing some degree $d$ it gives an upper bound on $r$ such that $C$ has a complete $g^r_d$.
It is known that $\Cliff (C) \geq 0$.
In \cite {ref18} a curve $C$ is called divisorial complete if the only restrictions on $r$ and $d$ for the existence of a complete $g^r_d$ on $C$ comes from the Riemann-Roch Theorem and the value of $\Cliff (C)$.
For a fixed Clifford index $c \geq 3$ it is proved in \cite {ref18} that $C$ cannot be divisorial complete if $g >2c+4$. 

During the recent decades, a theory of divisors on metric graphs is developed having lots of properties similar to those on smooth projective curves.
A divisor $D$ on a metric graph $\Gamma$ has a degree $d$ and a rank $r$ and there is a Riemann-Roch Theorem completely analogue to the case of curves.
In particular, given a fixed degree $d$, it gives the same lower bound on the rank $r$ as in the case of curves.
Also on a metric graph $\Gamma$ one defines similar sets $W^r_d(\Gamma)$ equal to the set of equivalence classes of divisors $D$ on $\Gamma$ of degree $d$ and rank at least $r$.
It has a natural polyhedral set structure, so we can speak of  $\dim (W^r_d(\Gamma))$ (see \cite{ref2}).
In \cite{ref2} the authors give an example of a non-hyperelliptic metric graph $\Gamma$ with $\dim (W^1_3(\Gamma))=1$, so this contradicts the statement of H. Martens' Theorem for metric graphs.

In the same paper the authors introduce so-called Brill-Noether numbers $w^r_d(\Gamma)$ (see Definition \ref{defBNnumber}) such that for a similar definition in the case of smooth curves $C$ one would have $w^r_d(C)=\dim(W^r_d(C))$.
However in the case of metric graphs those numbers have a better behaviour than $\dim(W^r_d(\Gamma))$.
In particular the mentioned example from \cite{ref2} satisfies $w^1_3(\Gamma)=0$.
So it is natural to study the statement of H. Martens' Theorem for metric graphs using those Brill-Noether numbers $w^r_d(\Gamma)$.
By defintion $w^r_d(\Gamma)\geq 0$ is equivalent to $W^r_d(\Gamma) \neq \emptyset$, so $\Gamma$ being hyperelliptic can be defined by the condition $w^1_2(\Gamma)\geq 0$.

In \cite{ref1} we study H. Martens' Theorem for chains of cycles (see Definition \ref{def1}).
We proved that the only counterexamples to the statement of H. Martens' Theorem in the case of chains of cycles are of a very special type and are called Martens-special chains of cycles in \cite{ref1} (see Definition \ref{def5}).
Those Martens-special chains of cycles are non-hyperelliptic and they satisfy $w^1_{g-1}(\Gamma)=g-3$.
Such a Martens-special chain of cycles $\Gamma$ has a type $k \geq 1$ and in \cite{ref1} it is proved that $\dim (W^1_{k+2}(\Gamma))=k$.
Now we study more closely so-called general Martens-special chains of cycles (see Definition \ref{def6}) obtaining the following theorem.

\begin{TheoremA}\label{theoremA}
If $\Gamma$ is a general Martens-special chain of cycles of type k then $w^1_{k+2}(\Gamma)=0$.
In particular $\Gamma$ has gonality $k+2$.
\end{TheoremA}
The gonality of a metric graph $\Gamma$ is the smallest integer $n$ such that $W^1_n(\Gamma) \neq \emptyset$.

This theorem illustrates two facts worth reporting.
\begin{itemize}
\item It illustrates that for metric graphs $\Gamma$ the discrepancy between $w^r_d(\Gamma)$ and $\dim (W^r_d(\Gamma))$ can become as large as possible.
\item It illustrates that metric graphs $\Gamma$ contradicting the statement of H. Martens' Theorem can have much larger gonality than 2.
In particular it illustrates that the weak H. Martens' Theorem for chains of cycles (Proposition 2 in \cite{ref1}) is, despite its weakness, a sharp theorem and that ''large'' values for Brill-Noether numbers do not in general imply ''small'' gonality in the case of metric graphs.
\end{itemize}

In \cite{ref5} one introduces the gonality sequence $g_r(C)$ (with $r \in \mathbb{Z}_{\geq 1}$) of a smooth curve $C$ with $g_r(C)$ being the minimal integer $n$ such that $W^r_n(C) \neq \emptyset$.
In particular $g_1(C)$ is the gonality of $C$.
Afterwards the gonality sequence became an object of study in the theory of divisors on smooth curves (see e.g. \cite{ref6}).
Of course, in the same way one can define gonality sequences of metric graphs and of finite graphs (see Definition \ref{def7}).
Those gonality sequences are studied for complete graphs in \cite{ref7} and for bipartite graphs in \cite{ref8} obtaining nice similarities with the gonality sequences of certain special related smooth curves.
A more general systematic study of gonality sequences of finite graphs is started in e.g. \cite{ref9} and \cite{ref10}.
Now we prove that general Martens-special chains of cycles have very special gonality sequences.

\begin{TheoremB}
If $\Gamma$ is a general Martens-special chain of cycles of type $k$ and genus $g$ then its gonality sequence is given by
\begin{description}
\item $g_r(\Gamma)=k+2r$ for $1 \leq r \leq g-k-1$
\item $g_r(\Gamma)=r+g-1$ for $g-k \leq r \leq g-1$
\item $g_r(\Gamma)=r+g$ for $r \geq g$
\end{description}
\end{TheoremB}

Also this theorem illustrates two facts worth reporting.
\begin{itemize}
\item For $k \geq 5$ such gonality sequence cannot occur for smooth curves if $g$ is large enough.
It should be noted that in \cite{ref9} and \cite{ref10} one also finds the existence of gonality sequences for finite graphs that cannot occur for smooth curves.
\item Taking into account the results of \cite{ref12} one finds lower bounds for $g_r(\Gamma)$ in case $\Gamma$ is a chain of cycles with fixed value $g_1(\Gamma)=d$.
Those lower bounds are obtained by taking into account the Riemann-Roch Theorem and the Clifford index of $\Gamma$ and nothing else (see Lemma \ref{lemmaE1}).
The theorem gives examples showing those lower bounds are sharp.
\end{itemize}

More generally we prove that for a fixed degree $d$ on a general Martens-special chain of cycles $\Gamma$ for each value $r$ satisfying the Riemann-Roch Theorem and the value of the Clifford index of $\Gamma$ there exists a divisor $D$ of degree $d$ and rank $r$ on $\Gamma$ (see Theorem \ref{theoremdivisorialComplete}). This means a general Martens-special chain of cycles of type $k$ is divisorial complete of Clifford index k (see Definition \ref{DivCompleteMetricGraph}).
Since the genus of $\Gamma$ can be taken arbitrary large this is again very different from the case of curves.

Since the discrete case of a finite graph is interesting too, we prove a similar result for finite graphs.
We introduce discrete chains of cycles and we define general Martens-special discrete chains of cycles of type $k$.
We obtain the following discrete version of Theorem B.

\begin{TheoremC}
If $G$ is a general Martens-special discrete chain of cycles of type $k$ and genus $g$ then its gonality sequence is given by
\begin{description}
\item $g_r(G)=k+2r$ for $1 \leq r \leq g-k-1$
\item $g_r(G)=r+g-1$ for $g-k \leq r \leq g-1$
\item $g_r(G)=r+g$ for $r \geq g$
\end{description}
\end{TheoremC}

In Section \ref{section2} we mention some generalities needed in the arguments of this paper.
In particular we make intensive use of some results from \cite {ref11}.
We also recall the definition of a Martens-special chain of cycles and we introduce the concept of a general Martens-special chain of cycles.

In Section \ref{section3} we prove that a general Martens-special chain of cycles of type $k$ has gonality $k+2$ (see Proposition \ref{prop1}).
The proof of this result is very short in a more special case (but including arbitrary large genus).
For a reader only interested in seeing arguments for the statements mentioned at the beginning of the introduction, this is sufficient.
The proof including all cases is more complicated (but contains no new ideas).
This is also the case for the proof of Theorem A in Section \ref{section5}, which is a continuation of the proof of Proposition \ref{prop1}.
Now the details of the proof in general become much more complicated (called modifications of the proof in the general case).
At the end one case (numbered (5)) is written down in almost all details, not all details for the other cases are written down.
Also after the proof of Proposition \ref{prop1}, the Lemmas \ref{lemma7}, \ref{lemma8}, \ref{lemma9} and \ref{lemma10} are only written down for usage in those modifications of the proof in the general case of Theorem A.

In Section \ref{section4} we recall the definition of gonality sequence (see Definition \ref{def7}) and the concept of a divisorial complete metric graph (see Definition \ref{DivCompleteMetricGraph}).
We give a proof of Theorem B and we prove that a general Martens-special chain of cycles is divisorial complete (see Theorem \ref{theoremdivisorialComplete}).

In Section \ref{section5} we prove Theorem A.
Since the theory of $v$-reduced divisors on metric graphs is essential in the proof of Theorem A, we start by recalling that theory (it was already used in the proof of Theorem \ref{theoremdivisorialComplete}).
As mentioned before, we first give a proof in a special case (which is already more involved than the proof of Proposition \ref{prop1} on the gonality) and for making the statement complete we give the modifications for the proof in the general case (but this is not needed any more to illustrate the statements at the beginning of the introduction).

In Section 6 we give a proof of Theorem C.
This is included because in the literature (see e.g. \cite{ref9} and \cite{ref10}) much attention goes to gonality sequences of finite graphs $G$.
In particular Theorem C shows that fixing the Clifford index of a finite graph $G$ the gonality sequence can be as strong as possible (having smallest possible values).
In order to proof Theorem C we first need to define a general Martens-special discrete chain of cycles (see Definition \ref{MartensSpecialdiscrete}).
For the proof we also need to translate some results from \cite{ref11} to the case of discrete chains of cycles.
Although I am sure that specialists know this translation it seems not being available in the literature, therefore it is included in Section \ref{section6}.

\section{Generalities}\label{section2}

Since this paper is a sequel to \cite{ref1}, we refer to that paper for some terminology concerning metric graphs.

\begin{definition}\label{def1}
A chain of cycles $\Gamma$ of genus $g \geq 1$ is a metric graph consisting of $g$ cycles $C_1, \cdots, C_g$ that are connected for $1 \leq i \leq g-1$ by a segment going from $w_i \in C_i$ to $v_{i+1} \in C_{i+1}$.
We also assume there is some $v_1 \in C_1$ and $w_g \in C_g$ and $v_i \neq w_i$ for $1 \leq i \leq g$.
\end{definition}
We will use those notations $(C_i,v_i,w_i)$ for a chain of cycles throughout the paper.

\begin{figure}[h]
\begin{center}
\includegraphics[height=2 cm]{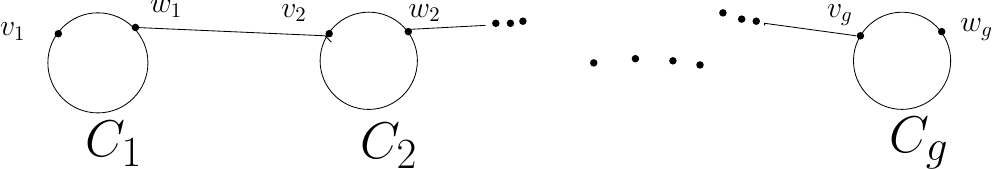}
\caption{a chain of cycles }\label{Figuur 1}
\end{center}
\end{figure}

Further on, if we talk about the union $C_1 \cup \cdots \cup C_i$ for some $1 < i < g$ then we mean this union including the segments $[w_j, v_{j+1}]$ for $1 \leq j < i$.
We also assume some orientation is chosen on each cycle $C_i$.
We recall the definition of the torsion profile $\underline{m}=(m_2, \cdots , m_g)$ of a chain of cycles $\Gamma$ coming from \cite{ref11}, Definition 1.9.

\begin{definition}\label{def2}
Let $l_i$ be the lenght of the cycle $C_i$ and let $l(v_i,w_i)$ be the lenght of the positive oriented arc on $C_i$ from $v_i$ to $w_i$.
In case $l_i$ is an irrational multiple of $l(v_i,w_i)$ then $m_i=0$.
Otherwise $m_i$ is the minimal positive integer such that $m_i.l(v_i,w_i)$ is an integer multiple of $l_i$.
\end{definition}

Following \cite{ref11}, Definition 3.1, we use the following convention to denote points on $C_i$.

\begin{definition}\label{def3}
For $\xi \in \mathbb{R}$ let $<\xi >_i$ denote the point on $C_i$ that is located $\xi .l(v_i,w_i)$ units from $w_i$ using a positive oriented path.
\end{definition}

It follows that, for integers $n_1$,$n_2$ one has $<n_1>_i=<n_2>_i$ if and only if $n_1 \equiv n_2 \mod{m_i}$ ($m_i$ as in Definition \ref{def2}).
Note also that $<0>_i=w_i$ and $<-1>_i=v_i$.

We refer to Section 1 of \cite{ref1} for definitions of divisors and their ranks on metric graphs.
From \cite{ref11}, Lemma 3.3, we obtain

\begin{lemma}\label{lemma1}
Let $D$ be any divisor of degree $d$ on a chain of cycles of genus $g$.
Then $D$ is equivalent to a unique divisor of the form $\sum_{i=1}^g <\xi _i>_i +(d-g).w_g$.
\end{lemma}

We are going to call this unique divisor from Lemma \ref{lemma1} the representing divisor of $D$.
The following definition also comes from \cite{ref11}.

\begin{definition}\label{def4}
For positive integers $m$ and $n$ we write $[m \times n]$ to denote the set $\{ 1, \cdots , m\}\times \{1, \cdots, n\}$.
It is represented by a rectangle with $n$ rows and $m$ columns.
The element on the $i$-th column and the $j$-th row corresponds to $(i,j) \in [m\times n]$.

\begin{figure}[h]
\begin{center}
\includegraphics[height=2 cm]{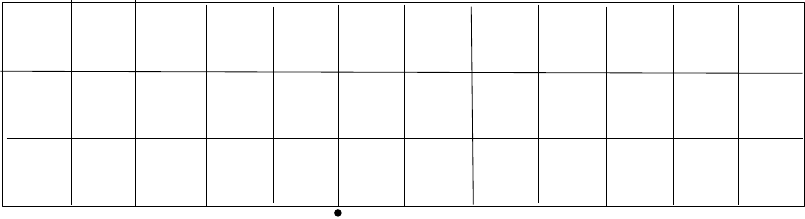}
\caption{rectangle $[12 \times 3]$ }\label{Figuur 2}
\end{center}
\end{figure}

Let $\underline{m}$ be a possible torsion profile of a chain of cycles of genus $g$.
An $\underline{m}$-displacement tableau on $[m \times n]$ of genus $g$ is a function $t:[m \times n] \rightarrow \{ 1, \cdots ,g\}$ such that
\begin{enumerate}
\item $t$ is strictly increasing in any given row and column of $[m \times n]$.
\item if $t(x,y)=t(x',y')$ then $x-y \equiv x'-y' \mod{m_{t(x,y)}}$.
\end{enumerate}
\end{definition}

We are going to make intensive use of the following statement which is part of Theorem 1.4 in \cite{ref11}.

\begin{theorem}\label{theorem1}
Let $\Gamma$ be a chain of cycles of genus $g$ with torsion profile $\underline{m}$.
A divisor $D$ on $\Gamma$ of degree $d$ has rank at least $r$ if and only if on $\lambda = [(g-d+r) \times (r+1)]$ there exists an $\underline{m}$-displacement tableau $t$ such that for the representing divisor $\sum_{i=1}^g <\xi _i>_i +(d-g)w_g$ of $D$ one has $\xi _{t(x,y)} \equiv x-y \mod{m_{t(x,y)}}$ for all $(x,y) \in \lambda$.
\end{theorem}

In \cite{ref1} we introduced the so-called Martens-special chains of cycles.

\begin{definition}\label{def5}
A chain of cycles $\Gamma$ of genus $g$ is called a Martens-special chain of cycles of type $k$ if and only if there exist integers $3 \leq j_1 \leq j_2 \leq \cdots \leq j_k \leq g-2$ with $j_{i+1}-j_i\geq 2$ for $1 \leq i \leq k-1$ such that
\begin{description}
\item $m_i \neq 2$ for $i\in \{ j_1, \cdots, j_k \}$
\item $m_i = 2$ for $i \in \{ 2, \cdots , g \} \setminus \{j_1 , \cdots , j_k \}$
\end{description}
\end{definition}
We will use those notations $j_1, \cdots, j_k$ for a Martens-special chain of cycles of type $k$ throughout this paper.

In \cite{ref1} it is proved that those are exactly the chains of cycles $\Gamma$ that are not hyperelliptic although $w^1_{g-1}(\Gamma)=g-3$.
It is also proved in \cite{ref1} that $\dim (W^1_{k+2}(\Gamma))=k$ for such Martens-special chains of cycles.
In this paper we study in more detail the "general" case.

\begin{definition}\label{def6}
A Martens-special chain of cycles $\Gamma$ is called general if $m_i=0$ in case $m_i \neq 2$.
\end{definition}

\section{The gonality of a general Martens-special chain of cycles}\label{section3}

\begin{proposition}\label{prop1}
Let $\Gamma$ be a general Martens-special chain of cycles of type $k$.
Then the gonality of $\Gamma$ is equal to $k+2$.
\end{proposition}

In the proof of Proposition \ref{prop1} we are going to find restrictions on the $\underline{m}$-displacement tableaux $t$ of genus $g$ on $[(g-k-1) \times (2)]$.
Using those restrictions, in Section \ref{section5}, we are going to prove the (a priori) stronger statement that $w^1_{k+2}(\Gamma)=0$.
In both proofs we first give the proof in case $j_{i+1}-j_i \geq 3$ for all $1 \leq i \leq k-1$.
In case those inequalities do not hold the proofs become more complicated.

\begin{proof}[Proof of Proposition \ref{prop1}]
As noticed in \cite{ref13}, each $\underline{m}$-displacement tableau of genus $g$ on $[(g-l-1) \times (2)]$ for some integer $l \geq 0$ can be obtained from the hyperelliptic tableau on $[(g-1) \times (2)]$ of genus $g$ by deleting $l$ values in each row.
This hyperelliptic tableau is defined by $t(i,j)=i+j-1$ for $(i,j) \in [(g-1) \times 2]$.
Of course, we have to take care it satisfies the conditions of an $\underline{m}$-tableau.
Since for $1 \leq i \leq k$ we have $m_{j_i}=0$, each value $j_i$ has to be deleted at least once from the hyperelliptic tableau. (In what follows by saying "delete" we mean from the hyperelliptic tableau.)

Assume we delete some $j_i$ only once.
If we would not delete $j_i-1$ neither $j_i+1$ at least once then since $m_{j_i-1}=m_{j_i+1}=2$, writing $t(a,1)=t(a',2)=j_i-1$ we need $a-a' \equiv 1 \mod 2$.
One possibility is $t(a+1,1)=t(a'+2,2)=j_i+1$, the other possibility is $t(a+2,1)=t(a'+1,2)=j_i+1$.
Both possibilities contradict $m_{j_i+1}=2$.
We conclude that $j_i-1$ or $j_i+1$ has to be deleted at least once.

\underline{End of the proof in case $j_{i+1}-j_i \geq 3$ for each $1 \leq i \leq k-1$.}

We obtain that at least $2k$ values have to be deleted.
So we cannot have $l<k$, therefore the gonality of $\Gamma$ is equal to $k+2$.
(Remember, in \cite {ref1} it is already proved that $W^1_{k+2}(\Gamma) \neq \emptyset$.)

\underline{Continuation of the proof in general.}

Now assume there exist $1 \leq i_1 < i_2 \leq k$ such that for $i_1 \leq i < i_2$ we have $j_{i+1}=j_i+2$ and assume for each $i_1 \leq i \leq i_2$ the image of $t$ contains $j_i$.
Let $i_1 \leq i <i_2$.
In the previous argument we can delete the value $j_i+1$ to compensate for retaining one value $j_i$ and one value $j_{i+1}$ while deleting no value $j_i-1$ and no value $j_{i+1}+1$.

Instead we are going to use the following claim.

\begin{claim}
If $i_2-i_1=m$ then at least $m+1$ values belonging to $\{ j_{i_1}-1, j_{i_1}+1,\cdots, j_{i_2}+1 \}$ are deleted.
(Some of those values can be deleted twice.)
\end{claim}

This implies we need to delete at least $2k$ values from the hyperelliptic graph of genus $g$, hence $l<k$ cannot occur.
So also in the general case we obtain $\Gamma$ has gonality $k+2$.
\end{proof}

\begin{proof}[Proof of the claim]

First we assume $m=1$.
Assume $j_{i_1}+1$ is not deleted.
Then to compensate the fact that $j_{i_1}$ and $j_{i_2}$ occur in $t$ we find $j_{i_1}-1$ and $j_{i_2}+1$ have to be deleted once, proving the claim in that case.
Assume $j_i+1$ is deleted exactly once and assume neither $j_{i_1}-1$ nor $j_{i_2}+1$ are deleted.
Assume $t(a_1,1)=t(a'_1,2)=j_{i_1}-1$.
Since $m_{j_{i_1}-1}=2$ we have $a_1-a'_1 \equiv 1 \mod {2}$.
Let $t(a_2,1)=t(a'_2,2)=j_{i_2}+1$.
Since $m_{j_{i_2}+1}=2$ we need also $a_2-a'_2 \equiv 1 \mod {2}$.
But the three values $j_{i_1}$, $j_{i_1}+1$ and $j_{i_2}$ are the only values occuring between $j_{i_1}-1$ and $j_{i_2}+1$ and each of them occurs exactly once. 
This implies $a_1-a'_1$ and $a_2-a'_2$ cannot have the same parity, giving a contradiction.
We obtain that one of the values belonging to $\{j_{i_1}-1, j_{i_2}+1 \}$ is deleted too.
We conclude that at least two values belonging to $\{ j_{i_1}-1, j_{i_1}+1, j_{i_2}+1 \}$ are deleted.

Now we assume $m\geq 2$ and the claim holds for $m-1$ (induction hypothesis).
From the induction hypothesis on $\{ j_{i_1}, \cdots , j_{i_2-1} \}$  it follows at least $m$ values of $\{ j_{i_1}-1, j_{i_1}+1, \cdots , j_{i_2}-1 \}$ are deleted.
In case more than $m$ values of $\{ j_{i_1}-1, j_{i_1}+1, \cdots , j_{i_2}-1 \}$ are deleted then the statement holds for $\{ j_{i_1} \cdots , j_{i_2} \}$ so we can assume exactly $m$ values of $\{ j_{i_1}-1, j_{i_1}+1, \cdots j_{i_2}-1 \}$ are deleted.
Assume no value $j_{i_2}+1$ is deleted.
From the induction hypothesis we know at least $m$ values from $\{ j_{i_1}+1, \cdots , j_{i_2}+1 \}$ are deleted.
Combining the previous facts we find that exactly $m$ values of $\{ j_{i_1}+1, \cdots, j_{i_2}-1 \}$ are deleted, but no value $j_{i_1}-1$ and $j_{i_2}+1$ is deleted.
We have $t(a_1,1)=t(a'_1,2)=j_{i_1}-1$ with $a_1-a'_1 \equiv 1 \mod {2}$ and $t(a_2,1)=t(a'_2,2)=j_{i_2}+1$ with $a_2-a'_2 \equiv 1 \mod {2}$.
However between $j_{i_1}-1$ and $j_{i_2}+1$  the $m+1$ values $j_k$ with $i_1 \leq k \leq i_2$ occur exactly once in $t$ and there are exactly $m$ values from $\{ j_{i_1}+1, \cdots ,j_{i_2}-1 \}$ in the tableau $t$.
So $a_1-a'_1$ and $a_2-a'_2$ cannot have the same parity.
This proves in this case some value $j_{i_2}+1$ has to be deleted.
This finishes this induction proof.

\end{proof}

\begin{remark}

1) If one is only interested in showing that the weak H. Martens' Theorem for chains of cycles (see \cite{ref1}, Proposition 2) is sharp, it is enough to restrict to Martens-special chains of cycles satisfying $j_{i+1}-j_i \geq 3$ for $1 \leq i \leq k-1$.
Then one can skip the more detailled arguments from the proof.

2) In the proof, in case $1 \leq i \leq k$, we only use that $j_i$ cannot occur twice in an $\underline{m}$-displacement tableau on $[(g-l+1) \times 2]$.
This follows from the assumption $m_{j_i}=0$.
However it would also follow from the condition $m_{j_i}> \min \{2j_i-2, 2g-2j_i \}$.
In particular it would also follow from $m_{j_i} \geq 2g$ for all $1 \leq i \leq k$. 

\end{remark}

Further on we use the following lemma that follows from the proof of Proposition \ref{prop1}.
\begin{lemma}\label{lemma7+}
If $t$ is an $\underline{m}$-displacement tableau on $[(g-k-1) \times 2]$ with $\underline{m}$ being the torsion profile of a Martens-special chain of cycles of genus $g$ and of type $k$ then $1, g \in \im t$ (hence $t(1,1)=1$ and $t(g-k-1,2)=g$).
\end{lemma}

\begin{proof}
From the proof of Propostion \ref{prop1} it follows that, if some $i \notin \{ j_1, \cdots, j_k \}$ is deleted then $i=j_{i'}-1$ or $i=j_{i'}+1$ for some $1 \leq i' \leq k$.
However from the definition of a Martens-special chain of cycles, it follows $j_1 \geq 3$ and $j_k \leq g-2$.
Therefore 0 and $g$ are not deleted.
\end{proof}

In order to prove Theorem A we need to make some further restrictions on the $\underline{m}$-displacement tableaux $t$ on $[(g-k-1) \times 2]$ related to the discussion in the continuation of the proof in general of Proposition \ref{prop1}.
In particular these consideration are not needed for the proof of Theorem A in case $j_{i+1}-j_i \geq 3 $ for $1 \leq i < k$.

In the next lemmas we assume $1 \leq i_1 <  i_2 \leq k$ is as in the continuation of the proof in general of Proposition \ref{prop1} and assume the set $\{ i_1, \cdots , i_2 \}$ is maximal of that kind.
\begin{lemma}\label{lemma7}
Exactly $i_2-i_1+1$ values from $\{ j_{i_1}-1, \cdots , j_{i_2}+1 \}$ are deleted (again, possibly some of them twice).
\end{lemma}
\begin{proof}
From the claim in the proof of Proposition \ref{prop1} we know at least $i_2-i_1+1$ values are deleted.
If it would be more one cannot obtain a tableau on $[(g-k-1) \times 2]$ any more because too many values are deleted.
\end{proof}

\begin{lemma}\label{lemma8}
The values $j_{i_1}-1$ and $j_{i_2}+1$ cannot be deleted twice.
\end{lemma}
\begin{proof}
If $j_{i_1}-1$ is deleted twice then because of Lemma \ref{lemma7} only $i_2-i_1-1$ values from $\{ j_{i_1}+1, \cdots , j_{i_2}+1 \}$ are deleted, contradicting the claim in the proof of Proposition \ref{prop1} for $\{j_{i_1+1}, \cdots , j_{i_2} \}$.
The argument in case $j_{i_2}+1$ is deleted twice is similar.
\end{proof}

\begin{lemma}\label{lemma9}
Assume $i_1 \leq l <i_2$ such that $j_l+1$ is deleted twice.
Then $j_l-1$ and $j_{l+1}+1$ are not deleted twice.
\end{lemma}

\begin{proof}
In case $l=i_1$ or $l=i_2$ then this follows from Lemma \ref{lemma8}.
So assume $i_1+1 \leq l <i_2-1$.
From the claim in the proof of Proposition \ref{prop1} we know $l-i_1+x$ ($x \geq 0$) values of $\{ j_{i_1}-1 , \cdots , j_l-1 \}$ and $i_2-l-1+y$ ($y\geq 0$) values of $\{ j_{l+1}+1, \cdots , j_{i_2}+1 \}$ are deleted.
Since $(l-i_1+x)+(i_2-l-1+y)+2$ has to be equal to $i_2-i_1+1$ we find $x=y=0$.
This implies $j_l-1$ (resp. $j_{l+1}+1$) cannot be deleted twice because otherwise only $l-i_1-2$ values of $\{j_{i_1}-1, \cdots , j_{l-1}-1 \}$ (resp. $i_2 -l-3$ values of $\{ j_{l+2}+1, \cdots , j_{i_2}+1 \}$) are deleted, contradicting the claim in the proof of Proposition \ref{prop1}.
\end{proof}

 Let $i_1 \leq i'_1 < i'_2  < \cdots < i'_n < i_2$ be defined by $j_{i'}+1$ is deleted twice for some $i_1 \leq i' \leq i_2$ if and only if $i'=i'_l$ for some $1 \leq l \leq n$.
 From the proof of the previous lemma we obtain that exactly $i'_1-i_1$ values of $\{ j_{i_1}-1 , \cdots , j_{i'_1}-1 \}$; exactly $i'_l-i'_{l-1}-2$ values of $\{ j_{i'_{l-1}+1}+1, \cdots , j_{i'_l}-1 \}$ for $2 \leq l \leq n$ and exactly $i_2-i'_n-1$ values of $\{j_{i'_n+1}+1, \cdots , j_{i_2}+1 \}$ are deleted.
Therefore in each of those sets at least one element is not deleted at all.
Since no other element of those sets is deleted twice it follows in each set this element is unique.

We obtain the following result.
\begin{lemma}\label{lemma10}
\begin{enumerate}
\item Let $i_1 \leq i'_1 <i'_2 \leq i_2-1$ be such that $j_{i'_1}+1$ and $j_{i'_2}+1$ both do not occur in $t$ and each value $j_{i'_1}+1<i<j_{i'_2}+1$ occurs in $t$ then there is a unique $i'_1<i'<i'_2$ such that $j_{i'}+1$ occurs twice in $t$.
\item Let $i_1 \leq i'_1 < i'_2 \leq i_2+1$ be such that $j_{i'_1}-1$ and $j_{i'_2}-1$ both occur twice in $t$ and no value $j_{i'_1}-1<i<j_{i'_2}-1$ occurs twice in $t$ then there is a unique $i'_1 < i' < i'_2$ such that $j_{i'}-1$ does not occur in $t$.
\item Let $i_1 \leq i' \leq i_2-1$ be such that $j_{i'}+1$ does not occur in $t$ and each value $j_{i_1}-1\leq i \leq j_{i'}$ (resp. $j_{i'+1}\leq i \leq j_{i_2}+1$) occurs in $t$ then there is a unique $i_1 \leq i'' \leq i'$ (resp. $i'+1 \leq i'' \leq i_2$) such that $j_{i''}-1$ (resp. $j_{i''}+1$) occurs twice in $t$.
\end{enumerate}
\end{lemma}

\section{Proof of Theorem B}\label{section4}

The topic of this section concerns gonality sequences, a subject coming from the theory of algebraic curves (see e.g. \cite{ref5}, \cite{ref6}) and adapted to graphs (see e.g. \cite{ref9}, \cite{ref10}.

\begin{definition}\label{def7}
Let $\Gamma$ be a metric graph.
The gonality sequence of $\Gamma$ is the sequence of integers $g_k(\Gamma)$ with $k \in \mathbb{Z}_{\geq 1}$ defined by  $g_k(\Gamma)= \min  \{ n\in \mathbb{Z} : \text{ there exists a divisor } D \text { on } \Gamma \text{ with } \deg (D)=n \text { and } \rank (D) \geq k \}$.
In particular $g_1(\Gamma)$ is the gonality $\gon (\Gamma)$ of $\Gamma$.
\end{definition}

As in the case of curves we define the Clifford index of a metric graph $\Gamma$ as follows.

\begin{definition}\label{CliffIndMetricGraph}
Let $\Gamma$ be a metric graph.
If $D$ is a divisor on $\Gamma$ of degree $d$ and rank $r \geq 0$ then its Clifford index is defined by $\Cliff (D)=d-2r$.
The Clifford index $\Cliff (\Gamma)$ of $\Gamma$ is the minimal value $\Cliff (D)$ for divisors $D$ of degree $d$ and rank $r$ satisfying $r \geq 1$ and $g-d+r-1 \geq 1$.
\end{definition}

We first prove the following general statement on gonality sequences of chains of cycles.

\begin{lemma}\label{lemmaE1}
Let $\Gamma$ be a chain of cycles of genus $g$ and let $g_1(\Gamma)=\gon (\Gamma)$ be the first number of the gonality sequence of $\Gamma$.
Then
\begin{description}
\item $g_r(\Gamma)=g+r$ in case $r \geq g$,
\item $g_r(\Gamma)=g-1+r$ in case $g-g_1(\Gamma)+1 \leq r \leq g-1$,
\item $g_1(\Gamma)+2r-2 \leq g_r(\Gamma) \leq g-1+r$ in case $1 \leq r \leq g-g_1(\Gamma)$.
\end{description}
\end{lemma}

\begin{proof}
From the Riemann-Roch Theorem for metric graphs it follows that $d_r(\Gamma)=g+r$ in case $r \geq g$.
For an effective divisor $E$ of degree $e$ on $\Gamma$ it also follows that $\rank (K_{\Gamma}-E) \geq g-e-1$ (here $K_{\Gamma}$ is the canonical divisor of $\Gamma$).
So we obtain $g_r(\Gamma) \leq 2g-2-(g-1-r)=g-1+r$ for $1 \leq r \leq g-1$.

In case $g_r(\Gamma)<g-1+r$ for some $1 \leq r \leq g-1$ then there exists a divisor $D$ of degree $d=g_r(\Gamma)<g-1+r$ with $\rank (D) \geq r$ and therefore $\rank (K_{\Gamma}-D) \geq 1$.
It is proved in \cite{ref12} that $\Cliff (\Gamma)=\gon (\Gamma)-2$ for a chain of cycles $\Gamma$.
For the above divisor $D$, by definition we have $d-2r=g_r(\Gamma)-2r \geq \Cliff (\Gamma) = g_1 (\Gamma)-2$ hence $g_r (\Gamma) \geq g_1 (\Gamma)+2(r-1)$.
Therefore, if $g_r (\Gamma)<g-1+r$ for some $1 \leq r \leq g-1$, we find $g-1+r > g_1(\Gamma) +2r-2$, hence $r< g-g_1 (\Gamma)+1$.
So for $g-g_1 (\Gamma)+1 \leq r \leq g-1$ we find $g_r (\Gamma)=g-1+r$.

Also in case $1 \leq r \leq g-g_1 (\Gamma)$, the previous argument gives $g_1 (\Gamma) +2(r-1) \leq g_r (\Gamma) \leq g-1+r$.
\end{proof}

We now give the proof of Theorem B, showing that the lower inequalities in Lemma \ref{lemmaE1} are sharp if $\Gamma$ is a general Martens-special chain of cycles.

\begin{proof}[Proof of Theorem B]
It is enough to prove that for each integer $r$ with $1 \leq r \leq g-k-2$ ther exists a divisor $D$ on $\Gamma$ of degree $k+2r=g_1 (\Gamma)+2(r-1)$ and rank $r$.
Because of Theorem \ref{theorem1} we need to prove that there exists an $\underline{m}$-displacement tableau on $[(g-(k+2r)+r) \times (r+1)]=[(g-k-r) \times (r+1)]$.

Write $\{ 1, \cdots, g \} \setminus \{j_1, \cdots , j_k \}=\{ l_1 < l_2 < \cdots < l_{g-k} \}$.
On $[(g-k-r) \times (r+1)]$ define $t_r(i,j)=l_{i+j-1}$.
Since $m_{l_i}=2$ for $2 \leq i \leq g-k-1$ this $t_r$ is an $\underline{m}$-displacemnt tableau on $[(g-k-r) \times (r+1)]$.
\end{proof}

As in the case of  curves we define the notion of a divisorial complete metric graph.

\begin{definition}\label{DivCompleteMetricGraph}
A metric graph $\Gamma$ of genus $g$ and Clifford index $c$ is divisorial complete if for all integers $d$ and $r$ there exists a divisor $D$ on $\Gamma$ of degree $d$ and rank $r$ if and only if it satisfies the restrictions coming from the Riemann-Roch Theorem on metric graphs and the Clifford index of $\Gamma$
\end{definition}

More concretely if $\Gamma$ is a metric graph of genus $g$ and Clifford index $c$ then $\Gamma$ if divisorial complete if and only if there exists a divisor $D$ on $\Gamma$ of degree $d$ and rank $r$ in each of the following cases.
\begin{itemize}
\item $d \geq g$ and $r=d-g$
\item $0 \leq d \leq g$ and $r=0$
\item $g \leq d \leq 2g-2$ and $r=d-g+1$
\item $r \geq 1$ and $c+2r \leq d \leq g+r-2$
\end{itemize}

The following Theorem is a sharper statement than Theorem B.

\begin{theorem}\label{theoremdivisorialComplete}
Let $\Gamma$ be a general Martens-special chain of cycles of type $k$ and genus $g$, then $\Gamma$ is a divisorial complete graph.
\end{theorem}
\begin{proof}
Because of the Theorem of Riemann-Roch, taking any divisor $D$ on $\Gamma$ of degree $d>2g-2$, it has rank $r=d-g$.
Notice that in case $c+2r \leq d \leq g+r-2$ we have $r \leq g-c-2$ and therefore $d \leq 2g-2-(c+2)<2g-2$.
So we can restrict to the case $d \leq 2g-2$.

For $1 \leq i \leq g$ choose $P_i \in C_i \setminus \{v_i,w_i \}$ and $P \in C_g \setminus \{v_g, P_g \}$.
For every integer $0 \leq d \leq g$ the divisor $D_d= \sum^d_{i=1}P_i$ is a $P$-reduced divisor (at the beginning of Section \ref{section5} this concept is shortly recalled).
Since $P$ does not belong to $D_d$ it implies $\rank (D_d)=0$.
So for each $0 \leq d \leq g$ we find a divisor of degree $d$ and rank 0.

Also for every integer $g+1 \leq d \leq 2g-2$ the divisor $D_d=\sum ^g_{i=1}P_i+(d-g)P$ is a $P$-reduced divisor.
This implies $\rank  (D_d) \leq d-g$ and the Riemann-Roch Theorem implies $\rank (D_d)=d-g$.

Using the Riemann-Roch Theorem for every integer $g \leq d \leq 2g-2$ we find $D'_d=K_{\Gamma}-D_{2g-2-d}$ has rank $r=d-g+1$.

So now we only have to prove in case $r \geq 1$, $g-d+r-1 \geq 1$ and $d-2r \geq k$  the existence of a divisor $D$ of degree $d$ and rank $r$ (remember $\Cliff (\Gamma)=k$).
We already noticed that $r \leq g-k-2$.

Letting $d'=\Cliff (\Gamma)+2r=k+2r$ with $1 \leq r \leq g-k-2$ then the integers $r$, $d'$ do satisfy the conditions and Theorem B implies there exist divisors $D'$ on $\Gamma$ of rank $r$ and degree $d'$.
From the proof of Theorem B and using the notations from that proof we can assume the representing divisor of $D'$ is of the type $D''=\sum ^g_{i=1}Q_i-(d'-g)w_g$ with $Q_{l_i} \in \{ v_{l_i},w_{l_i} \}$ and $Q_i$ arbitrary on $C_i$ if $i \in \{ j_1, \cdots , j_k \}$.

Since $d\leq g+r-2 \leq g+2r =g-k +(2r+k)$ we have $d-d' \leq g-k$.
Take general points $Q'_i$ on $C_{l_i}$ for $1 \leq i \leq d-d'$ and consider $D = D'+Q'_1 + \cdots +Q'_{d-d'}$.
Let $D'''=P'_1 + \cdots + P'_g + (d-g)w_g$ be the representing divisor of $D$.
Since $D'''$ is equivalent to $D''+Q'_1 + \cdots + Q'_{d-d'}$ there exists a rational function $f$ on $\Gamma$ such that $\divi (f)=(D''+Q'_1 + \cdots + Q'_{d-d'})-D'''$ (this concept is recalled in the introduction of \cite {ref1}).
For $1 \leq i \leq d-d'$ the restriction $f_i$ of $f$ to $C_{l_i}$ implies the existence of integers $a$ and $b$ such that on $C_{l_i}$ we have $\divi (f_i)=av_{l_i} + bw_{l_i} + Q'_i-P'_{l_i}$.
Since $Q'_i$ is general on $C_{l_i}$ we have $P'_{l_i} \notin \{ v_{l_i}, w_{l_i} \}$.
Also since we can take $Q_{j_i}$ general on $C_{j_i}$ we have $P'_{j_i}$ is general on $C_{j_i}$.

Now assume $\rank (D) >r$.
Because of Theorem \ref{theorem1} there exists an $\underline{m}$-displacement tableau $t$ on $\lambda = [(g-d+r+1) \times (r+2)]$ such that for all $(x,y) \in \lambda$ if $P'_{t(x,y)} = <\xi >_{t(x,y)}$ then $\xi \equiv x-y \mod {m_{t(x,y)}}$.
But for $1 \leq i \leq d-d'$ we have $m_{l_i}=2$ and $P'_{l_i} \notin \{ v_{l_i}, w_{l_i} \}$ hence $l_i \notin \im (t)$.
For $1 \leq i \leq k$ the point $P'_{j_i}$ is general on $C_{j_i}$ hence $j_i \notin \im (t)$ because of Lemma 3 in \cite {ref1}.
It follows that $\im (t)$ has at most $g-k-d+d'$ different values.
But a tableau on $\lambda$ needs at least $g-d+r+1+r+1=g-d+2r+2$ different values.
So we need $d'-k \geq 2r+2$ contradicting $d'=k+2r$.

Since $\rank (D')=r$ and $D-D'$ is effective is it follows $\rank (D) \geq r$.
So we conclude $D$ is a divisor of degree $d$ and rank $r$.
\end{proof}

In Lemma 1 of \cite{ref18} it is proved that a divisorial complete curve $C$ of genus $g$ satisfies the inequalities $2\Cliff (C)+1 \leq g \leq 2\Cliff(C) + 4$ in case $\Cliff(C) \geq 3$.
In particular if $g \geq 2k+4$ then there is no smooth curve of genus $g$ having the same gonality sequence as a general Martens-special chain of cycles of type $k$ of genus $g$ if $k\geq 3$.
Also, in contrast to the case of smooth curves, we obtain the existence of divisorial complete metric graphs of given Clifford index of arbitrary large genus.

At the moment no divisorial complete curve $C$ with $\Cliff(C) \geq 6$ is known to the author.

\section{Proof of Theorem A}\label{section5}

For a metric graph $\Gamma$ we recall the defintion of the Brill-Noether number introduced in \cite{ref2}.

\begin{definition}\label{defBNnumber}
Let $r$, $d$ be non-negative integers.
We define $w^r_d(\Gamma)=-1$ in case there is no divisor $D$ on $\Gamma$ of degree $d$ with $\rank (D) \geq r$.
Otherwise it is the maximal non-negative integer $w$ such that for each effective divisor $E$ of degree $w+r$ on $\Gamma$ there exists an effective divisor $D$ on $\Gamma$ of degree $d$ such that $D$ contains $E$ (meaning $D-E$ is an effective divisor) and $\rank (D) \geq r$.
\end{definition}

It should be noted that in case $C$ is a smooth irreducible projective curve and $W^r_d(C) \neq \emptyset$ then using a similar definition to define $w^r_d(C)$ using dimensions of complete linear systems, one has $w^r_d(C)=\dim (W^r_d(C))$.
 
In order to prove Theorem A we are going to use the theory of reduced divisors.
For a graph $\Gamma$ and a point $v$ on $\Gamma$ the definition of a $v$-reduced divisor on $\Gamma$ can be found in e.g. \cite{ref15}, Definition 2.2.
In \cite{ref15}, Theorem 2.3, it is  proved that for each divisor $D$ on $\Gamma$ there is a unique $v$-reduced divisor $D'$ equivalent to $D$ and Algoritm 2.12 of that paper describes how, starting with an effective divisor $D$, we can find this $v$-reduced divisor $D'$ equivalent to $D$.
From that algorithm it is clear that the coefficient of $v$ in $D'$ cannot be less than the coefficient of $v$ in $D$.
In particular if $F$ is a $v$-reduced divisor and $v$ is not a point of $F$ then there is no effective divisor $F'$ equivalent to $F$ and containing $v$.
In \cite{ref15}, Algorithm 5, one finds a method that can be used to decide whether $F$ is a $v$-reduced divisor.
This algorithm can be visualized as being Dhar's burning algorithm.
Image that $w \neq v$ contains as many backets of water as the coefficient of $w$ in $F$.
Then let a fire start at $v$ and this fire can be stopped at most once using a bucket of water.
The graph burns completely (we then say the fire starting at $v$ to $F$ burns the graph completely) if and only if $F$ is $v$-reduced.

\begin{proof}[Proof of Theorem A]

We are going to prove that each divisor of degree $k+2$ and rank at least 1 on $\Gamma$ is equivalent to an effective divisor $2v_g+F$ such that $F$ contains no point of $C_g$ and $F$ is a $v_g$-reduced divisor.
This suffices to finish the proof as follows.

On $C_g$ we can take $w_g$ such that there exists $w'_g \neq w_g$ with $2v_g$ is equivalent to $w_g + w'_g$.
Take $w''_g \in C_g \setminus \{v_g, w_g, w'_g \}$.
In case $w^1_{k+2}(\Gamma) \geq 1$ it would follow that there exists a $v_g$-reduced effective divisor $F$ of degree $k$ with no point on $C_g$ such that $w_g+w'_g+F$ is equivalent to $w_g+w''_g+F'$ for some effective divisor $F'$, hence there should exist an effective divisor on $\Gamma$ equivalent to $w'_g+F$ and containing $w''_g$.
However starting a fire on $\Gamma$ at $w''_g$ to $w'_g+F$, since $F$ contains no point of $C_g$ the fire arrives at $w'_g$ from two sides.
Since there is only one bucket of water at $w'_g$ it follows that also $w'_g$ gets burned.
But the fire also reaches $v_g$ and $F$ is $v_g$-reduced, therefore the fire burns the whole graph.
So $w'_g+F$ is a $w''_g$-reduced divisor and it does not contain $w''_g$, hence there is no effective divisor equivalent to $w'_g+F$ containing $w''_g$.
This gives a contradiction to $w^1_{k+2}(\Gamma) \geq 1$, hence $w^1_{k+2}(\Gamma)=0$.

We are now going to restrict the $\underline{m}$-displacement tableaux $t$ on $[(g-k-1) \times (2)]$ describing all elements of $W^1_{k+2}(\Gamma)$
Then we prove that each representing divisor obtained from such $t$ has an equivalent divisor $2v_g+F$ with $F$ not containing any point of $C_g$ and being $v_g$-reduced.

\underline{Proof in case $j_{i+1}-j_i \geq 3$ for all $1 \leq i < k$.}

In case $j_i$ is deleted exactly once and $j_i-1$ (resp. $j_i+1$) is deleted from the line still containing $j_i$, then replacing $j_i$ by $j_i-1$ (resp. $j_i+1$) in the tableau we still find the same points $<\xi _{i'}>_{i'}$ determined by the tableau for $i'\neq j_i$.
Indeed, from the even number of removes smaller than $j_i-1$, it follows for $t(x,y)=j_i-1$ and $t(x',y')=j_i$ one has $x-y \equiv x'-y' \mod{2}$.
So the mentioned replacement again gives rise to an $\underline{m}$-replacement tableau.
The only difference is  that the point on $C_{j_i}$ is not  fixed  any more by the tableau but can be chosen freely.
Therefore  the divisors coming from such tableau are also coming from the tableau obtained from such replacements.
Therefore, from now on, if $j_i$ is deleted exactly once we assume that $j_i-1$ or $j_i+1$ is deleted from the same line  as $j_i$.

This also implies that if exactly $x$ values $j_i$ ($1 \leq i \leq k$) appear on the first line of an $\underline{m}$-displacement tableau $t$ on $[(g-k-1) \times (2)]$ then also $x$ values $j_{i}$ appear on the second row.
Because of the tableaux-requirements, if before the $j$-th column the number of values $j_i$ on the first row is exactly $x'$ then at most $x'$ values $j_i$ are on the second row before the $j$-th column.

Let $D$ be a representing divisor coming from some $\underline{m}$-displacement tableau on $[(g-k-1) \times (2)]$.
Let $F'$ be the restriction of $D$ to $C_1 \cup \cdots \cup C_{g-1}$.
We are going to move $F'$ as much as possible "to the right".
By this we mean we look for an effective divisor $F+mw_{g-1}$ equivalent to $F'$ and $F$ the $v_g$-reduced divisor we are looking for.
Of course $mw_{g-1}$ can be replaced by $mv_g$.
In case each $j_i$ with $1 \leq i \leq k$ is deleted twice then $F=\xi _1 + \cdots + \xi_k$ with $\xi_i \in (C_{j_i} \setminus \{w_{j_i}\}) \cup \{ v_{j_i+1} \}$.
We will see that in case $j_i$ for some $1 \leq i \leq k$ appears in the first row of $t$ then there is no such point $\xi _i$ on $F$ .
In case $j_i$ appears in the second line of $t$ then we will see $\xi_i$ in $F$ is replaced by $\xi_i +v_{j_i+1}$ with $\xi_i \in C_{j_i} \setminus \{w_{j_i}\}$.
It will follow that $m=g-1-k$.
From Lemma \ref{lemma7+} it follows $D$ contains $<g-k-3>_g$.
On $C_g$ we have $(g-1-k)v_g+(k+2-g)w_g+<g-k-3>_g$ is equivalent to $2v_g$, therefore completing the proof in this case.

This moving "to the right" will be done step by step.
For $1 \leq i \leq g-1$ we write $F'_i$ to denote the restriction of $D$ to $C_1 \cup \cdots \cup C_i$ and we are going to prove it is equivalent to  a divisor of the type $F_i+m_iw_i$ with $m_i \geq 0$ and $F_i$ is a $v_g$-reduced divisor (it will be  the restriction of $F$ to $C_1 \cup \cdots  \cup C_i$).
Of course $mw_i$ can be replaced by $mv_{i+1}$.
We will say that "$mv_{i+1}$ is obtained from $F'_i$".
We are going to use induction on $i$ the case $i=1$ being true because of Lemma \ref{lemma7+}.

Take some $1 \leq i \leq g-1$.
Let $s$ (resp. $s'$; $s''$) be the number of $1 \leq i' \leq k$ such that $j_{i'}<i$ and $j_{i'}$ occurs on the second row of $t$ (resp. $j_{i'}$ occurs on the first row of $t$; $j_{i'}$ does not occur in $t$).
Note that because $t$ is a tableau we have $s \leq s'$.
By induction we are going to show that for $2 \leq i \leq g$ one has $(i-1-2s-s'')v_i$ is induced from $F'_{i-1}$ except from two cases.
If $i-1=j_{i'}$ and $j_{i'}$ does not occur in the tableau then there is the possibility that $F_{i-1}=F_{i-2}$ in which case $(i-2s-s'')v_i$ is induced from $F'_{i-1}$.
If $i-1=j_{i'}$ and $j_{i'}$ does occur in the second row then $(i-2-2s-s'')v_i$ is induced from $F'_i$.
For later use, let $\overline{s}$, $\overline{s'}$ and $\overline{s''}$ be similar numbers for the whole chain of cycles.

\begin{enumerate}
\item The case $i \notin \{j_1, \cdots, j_k \}$ and if $i=j_{i'}+1$ for some $1 \leq i' \leq k$ then $j_{i'}$ does not occur on the second row of $t$.

Depending on $i$ occuring in the first or second row of the tableau, from the tableau we have $<i-2s-1-s''>_i=<i-1-2s'-2-s''>_i$ in $F'$.
By induction we can assume that $(i-1-2s-s'')v_i$ is obtained from $F'_{i-1}$ (except for one case explained at the end of (2)).
In case $i-1-2s-s''$ is even then on $C_i$ we have $(i-1-2s-s'')v_i+w_i$ is equivalent to $(i-2s-s'')w_i$.
In case $i-1-2s-s''$ is odd then we obtain $(i-2s-s'')v_i$ and on $C_i$ this is equivalent to $(i-2s-s'')w_i$ on $C_i$.
In both cases $(i-2s-s'')v_{i+1}$ is  obtained from $F'_i$ and $F_i=F_{i-1}$.

\item The case $i=j_{i'}$ and $j_{i'}$ does not occur in the tableau.

By induction we can assume that $(i-1-2s-s'')v_i$ is obtained from $F'_{i-1}$.
On $F'$ we have the point on $C_i$ is arbitrary (we denote it by $\xi '_{i'}$).
There exists a point $\xi_{i'}$ on $C_i$ such that $(i-1-2s-s'')v_i+\xi'_{i'}$ is equivalent to $(i-1-2s-s'')w_i+\xi_{i'}$ on $C_i$.
Therefore $(i-2s-(s''+1))v_{i+1}$ is obtained from $F'_i$ and $F_i=F_{i-1}+\xi_{i'}$ in case $\xi_{i'} \neq w_i$.
In case $\xi_{i'}=w_i$ then in principle $(i+1-2s-(s''+1))v_{i+1}$ is obtained from $F'_i$ and $F_{i}=F_{i-1}$ (this is the case where the induction hypothesis in (1) does not hold).
But then in the proof of (1) we have to adapt the induction hypothesis and we find as it should be that 
$(i+1-2s-(s''+1))v_{i+2}$ is obtained from $F'_{i+1}$ but $F_{i+1}=F_{i-1}+v_{i+1}$.
\item The case $i=j_{i'}$ with $j_{i'}$ occurs in the first row of the tableau.

Again by induction we can assume $(i-1-2s-s'')v_i$ is obtained from $F'_{i-1}$.
Also the point $<i-2s-s''-1>_i$ belongs to $F'$.
On $C_i$ we have $<i-2s-s''-1>_i+(i-1-2s-s'')v_{i}$ is equivalent to $(i-2s-s'')w_i$.
Therefore $(i-2s-s'') v_{i+1}$ is obtained from $F'_i$ and $F_{i+1}=F_i$.

\item The case $i=j_{i'}$ and $j_{i'}$ occurs on the second row (this also settles the case $i=j_{i'}+1$ omitted in (1) in that case).

Again by induction we can assume $(i-1-2s-s'')v_i$ is obtained from $F'_{i-1}$.
On $F'$ we have the point $<i-2s'-s''-3>_i$.
On $C_i$ the divisor $(i-2s-s''-1)v_i+<i-2s'-s''-3>_i$ is equivalent to $<2(s-s')-2>_i+(i-2s-s''-1)w_i$.
Since $m_i=0$ and $s \leq s'$ we find $<2(s-s')-2>_i \in C_i \setminus \{w_i\}$.
This is going to be the point $\xi _{i'}$ on $F$.

From $F'_i$ we obtain $(i-2s-s''-1)v_{i+1}$ (we now use it to settle the omitted case in (1)).
The point $<i-2s'-s''-2>_{i+1}$ is on $F'$.
In case $i-2s-s''-1$ is even then this point on $F'$ is $v_{i+1}$.
Then on $C_{i+1}$ the divisor $(i-2s-s'')v_{i+1}$ is equivalent to $v_{i+1}+(i-2s-s''-1)w_{i+1}$.
In case $i-2s-s''-1$ is odd then this point on $F'$ is $w_{i+1}$.
Then on $C_{i+1}$ the divisor $(i-2s-s''-1)v_{i+1}+w_{i+1}$ is equivalent to $v_{i+1}+(i-2s-s''-1)w_{i+1}$.
In both cases we obtain $((i+1)-2(s+1)-s'')v_{i+2}$ is obtained from $F'_{i+1}$ and $F_{i+2}=F_i+\xi_{i'}+v_{i+1}$.

\end{enumerate}

\underline{Modifications of the proof in the general case}

In the modifications we have to take care of the following two facts.
\begin{itemize}
\item In the previous proof, for $1\leq i \leq k$, either $j_i$ does not appear in the tableau, else $(j_i, j_i+1)$ or $(j_i-1, j_i)$ are omitted once in the same row.
In the modification it is more subtle to choose the couples that have to be omitted in the same row such that the proof still works.
\item It can happen that for some $1 \leq i < k$ with $j_{i+1}-j_i=2$ the values $j_i$ and $j_{i+1}$ appear in the tableau while $j_i+1$ does not appear in the tableau.
In those cases the construction of the divisor $F$ in the proof is more subtle.
\end{itemize}

Assume we have $1 \leq i \leq k$ and $1 \leq n \leq k-i$ with $j_{i+l+1}-j_{i+l}=2$ for $0 \leq l <n$ and none of those values $j_{i+l}$ with $0 \leq l \leq n$ is deleted twice.
We assume $i$ and $n$ are taken such that this part of the chain can not be enlarged under those conditions.
This means in case $1<i$ then $j_{i-1}$ is deleted twice or $j_i - j_{i-1}>2$ and in case $i+n<k$ then $j_{i+n+1}$ is deleted twice or $j_{i+n+1}-j_{i+n} > 2$.

First we consider the case that no value from $\{j_i-1, j_i+1, \cdots, j_{i+n}+1 \}$ is deleted twice.
Then Lemma \ref{lemma7} implies exactly one  element of the set is not deleted.

In case $j_i-1$ is not deleted then we repeat the previous arguments using the couples $(j_i,j_i+1), \cdots , (j_{i+n}, j_{i+n}+1)$.
In particular now and further on this implies that if in such a couple both values are not on the same line then we  can change the tableau by replacing the first value of that couple by the second value.
In that new tableau we obtain smaller subchains of the type we are considering and we can assume such couples do not occur.
In that case the proof of cases (3) and (4) can be repeated.

Suppose $j_{i+l}+1$ is not deleted for some $1 \leq l \leq n$.
We repeat the previous arguments using $(j_i, j_i-1), \cdots , (j_{i+l},j_{i+l}-1), (j_{i+l+1},j_{i+l+1}+1), \cdots , (j_{i+n}, j_{i+n}+1)$.
In case for some $i \leq i' \leq i+l$ we have $j_{i'}$ is on the second row but $j_{i'}+1$ is on the first row, then we use (in (4)) $<j_{i'}-2s'-s''-2>_{j_{i'}+1} = <j_{i'}-2s-s''>_{j_{i'}+1}$ since $m_{j_{i'}+1}=2$.

Assume now some value of $\{ j_i-1, j_i+1, \cdots , j_{i+n}+1\}$ is deleted twice.
We already know it is some value $j_{i+l}+1$ with $0\leq l \leq n-1$ (Lemma \ref{lemma8}) and $j_{i+l}-1$ nor $j_{i+l+1}+1$ are deleted twice (Lemma \ref{lemma9}).
We are going to prove that, in case $j_{i+l}$ and $j_{i+l+1}$ occur on different rows of $t$ then as before those two values will contribute two points of $F$.
However if $j_{i+l}$ and $j_{i+l+1}$ are both on the second row then those two values contribute to three (instead of four) points on $F$ and in case $j_{i+l}$ and $j_{i+l+1}$ are both on the first row then those two points contribute to one (instead of no) point on $F$.
With respect to the couples we make the following rule.
In case $j_{i+l}-1$ (resp. $j_{i+l+1}+1$) does not appear twice then $(j_{i+l-1},j_{i+l-1}+1)$ (resp. $(j_{i+l+2},j_{i+l+2}-1)$) have to be considered as a couple.
In case for some $l'$  with $l+1 < l' <n$ we have $j_{i+l'}+1$ does not occur in the tableau (and $l'$  is as small as possible with this property) then from Lemma \ref{lemma10} we know there is some $i'$ between $i+l$ and $i+l'$ such that $j_{i'}+1$ occurs twice in the tableau.
A similar remark holds in case there is such $0 \leq l' <l$ using $j_{i+l'}-1$.
This makes it possible to change the $+$ and $-$ signs in the couples by omitting $j_{i'}+1$ in those couples.
Also in case $l'$ would not exist then there exists some $1 \leq i' \leq l$ (or $i+l+1 \leq i' \leq i+n$) such that $j_{i'}-1$ (resp. $j_{i'}+1$) occurs twice in the tableau.
The values $j_{i'}$ occuring in some couple can be handled using the previous arguments.
Indeed, the number of removals before each couple is even, so we can assume both values of a couple appear on the same row.
Therefore for each couple we can use case (3) or (4).

As an illustration, assume $n=10$ and $j_i=5$.
Assume $j_{i+2}+1=10$ and $j_{i+7}+1=20$ are deleted twice and $j_i-1=4$, $j_{i+5}+1=16$ and $j_{i+8}+1=22$ are not deleted.
Then the couples are $(5,6)$, $(7,8)$, $(13,12)$, $(15,14)$, $(17,18)$, $(23,24)$ and $(25,26)$.

The resulting divisor $F$ has degree $k$ because of the following calculation.
Let $\overline{v}$ (resp $\overline{u}$) be the number of indices $1\leq l \leq k$ such that $j_l+1$ is deleted twice and $j_l$ and $j_{l+1}$ are both on the second (resp first) row. 
Since the tableau has to be rectangular we have $2\overline{s'}-\overline{u}+\overline{v}=2\overline{s}-\overline{v}+\overline{u}$, so $\overline{s}-\overline{v}=\overline{s'}-\overline{u}$.
Therefore the divisor $F$ has degree $2\overline{s}-\overline{v}+\overline{u}+\overline{s''}=\overline{s}+\overline{s'}+\overline{s''}=k$.
Therefore we obtain the same final arguments given in the beginning of the proof, finishing the proof in general.

To finish the proof we assume $j_{i+l}+1$ as before is deleted twice and we need to prove the assertion on the contribution to the divisor $F$ according to the lines containing $j_{i+l}$ and $j_{i+l+1}$.
Let $s$, $s'$, $s''$ be as before and let $v$ (resp. $u$) be the number of values $j_{l'}<j_{i+l}$ such that $j_{l'}+1$ is deleted twice and $j_{l'}$ and $j_{l'+1}$ are both on the second (resp. first) row.
By induction we assume $(j_{i+l}-1-s''-2s+v-u)v_{j_{i+l}}$ will be obtained from $F'_{j_{i+l}-1}$.
We write $n_l=j_{i+l}-1-s''-2s+v-u$.
Since $t$ has to be a tableau we have $s+u \leq s'+v$.

\begin{enumerate}
\setcounter{enumi}{4}
\item $j_{i+l}$ is on the second row and $j_{i+l+1}$ is on the first row.

The point $\xi '_{i+l}=<j_{i+l}-2s'-s''-3+u-v>_{j_{i+l}}$ belongs to $F'$ and on $C_{j_{i+l}}$ we have $n_lv_{j_{i+l}}+\xi'_{i+l}$ is equivalent to $n_lw_{j_{i+l}}+\xi_{i+l}$ with $\xi_{i+l}=<2(s+u)-2(s'+v)-2>_{j_{i+l}}$.
In particular $\xi_{i+l}\notin \{v_{j_{i+l}}, w_{j_{i+l}}\}$.
So, from $F'_{j_{i+l}}$ we obtain $n_lv_{j_{i+l}+1}$ and $F_{j_{i+l}}=F_{j_{i+l}-1}+\xi _{i+l}$.

On $F'$ we have an arbitrary point $\tilde{\xi'}_{i+l}$ on $C_{j_{i+l}+1}$.
We obtain $\tilde{\xi}_{i+l}$ on $C_{j_{i+l}+1}$ such that $n_lv_{j_{i+l}+1}+\tilde{\xi'}_{i+l}$ is equivalent to $\tilde{\xi}_{i+l}+n_lw_{j_{i+l}+1}$.
In case $\tilde{\xi}_{i+l}\neq w_{j_{i+l}+1}$ (resp. $\tilde{\xi}_{i+l} = w_{j_{i+l}+1}$) then from $F'_{j_{i+l}+1}$ we obtain $n_lv_{j_{i+l+1}}$ (resp. $(n_l+1)v_{j_{i+l+1}}$) and $F_{j_{i+l}+1}=F_ {j_{i+l}}+\tilde{\xi}_{i+l}$ (resp. $F_{j_{i+l}+1}=F_{j_{i+l}}$).
Up to now we only used $j_{i+l}$ is on the second row.

The point $<n_l>_{j_{i+l+1}}$ is on $F'$.
On $C_{j_{i+l+1}}$ the divisor $n_lv_{j_{i+l+1}}+<n_l>_{j_{i+l+1}}$ is equivalent to $(n_l+1)w_{j_{i+l+1}}$ and $v_{j_{i+l}+1}$ is added to $F_{j_{i+l}+1}$ in case $\tilde{\xi}_{i+l}=w_{j_{i+l}+1}$.
Finally we show that in case $n_l$ is even (resp. odd) we have $v_{j_{i+l+1}+1}$ (resp. $w_{j_{i+l+1}+1}$) on $F'$, implying $n_{l+2}v_{j_{i+l+2}}$ is obtained from $F'_{j_{i+l+2}-1}$, so we can continue the proof. 

If we do not change anything to the tableau for values less than $j_{i+l}+1$ then it would not have been possible to obtain an $\underline{m}$-displacement tableau with 2 rows such that no value $j_{i+l}+1$ would have been deleted.
Indeed we know from Lemma \ref{lemma10} that there is some largest value $l'$ smaller than $l$ such that $j_{i+l'}+1$ occurs twice in the tableau.
But then if also $j_{i+l}+1$ appears twice in a tableau we know from Lemma \ref{lemma10} that there is some value $l'<l''<l$ such that $j_{i+l''}+1$ does not occur in the tableau.
But then from Lemma \ref{lemma10} we know that there is $l''<l'''<l$ such that $j_{i+l'''}+1$ appears twice in the tableau, contradiction the assumption on $l'$.
It implies that in case no value $j_{i+l}+1$ would have been deleted, this would give contradicting points in $F'$.

Assume now that $n_l$ is even (the other possibility needs similar arguments, those are left to the reader).
Since $j_{i+l}-1$ appears in the tableau it implies $v_{j_{i+l}-1}$ is in $F'$.
Indeed $n_lv_{j_{i+l}}$ is obtained from $F'_{j_{i+l}-1}$.
This means that while moving ''to the right'' we obtain a divisor containing $n_lw_{j_{i+l}-1}$.
We have the point of $F'$ on $C_{j_{i+l}-1}$ is equal to $w_{j_{i+l}-1}$ or $v_{j_{i+l}-1}$.
Since $m_{j_{i+l}-1}=2$ we have an even contribution to $n_l$ coming from $v_{j_{i+l}-1}$ while moving ''to the right''.
So in case $n_l$ is even there cannot be another contribution to $n_lw_{j_{i+l}-1}$, implying the claim.
If $j_{i+l}-1$ is on the second row then if $m_{j_{i+l}}$ would have been equal to 2, it would imply $w_{j_{i+l}}$ is on $F'$.
If $j_{i+l}-1$ is not on the second row then it has to be on the first row.
Also $l'<l-1$ ($l'$ as defined above) in this case.
As before this implies that if we do not change anything to the tableau for values less than $j_{i+l}-1$ then it would not have been possible to obtain an $\underline{m}$-displacement tableau with 2 rows such that no value $j_{i+l}-1$ is deleted.
So, if $j_{i+l}-1$ would have been on the second line on such $\underline{m}$-displacement tableau then it would imply $w_{j_{i+l}-1}$ on $F'$.
Since $j_{i+l}$ is on the second row instead of $j_{i+l}-1$ we find that if $m_{j_{i+l}}$ would have been equal to 2, it would imply $w_{j_{i+l}}$ is in $F'$.

So if $j_{i+l}+1$ would have been on the second row it would have given $v_{j_{i+l}+1}$ on $F'$.
This implies that if $j_{i+l+1}+1$ is on the second row it implies the point $v_{j_{i+l+1}+1}$ on $F'$.
If $j_{i+l}+1$ would have been on the first row it would have given $w_{j_{i+1}+1}$ on $F'$.
Therefore if $m_{j_{i+l+1}}$ would have been equal to 2, it would imply $w_{j_{i+l+1}}$ on $F'$.
This implies again that if $j_{i+l+1}+1$ is on the first row it implies the point $v_{j_{i+l+1}+1}$ on $F'$.
Since $j_{i+l+1}+1$ is a value of the tableau, we conclude that $v_{j_{i+l+1}+1}$ is on $F'$.  
\end{enumerate}

In the following cases we leave the details that are similar to those occuring in case (5) to the reader.

\begin{enumerate}
\setcounter{enumi}{5}
\item $j_{i+l}$ and $j_{i+l+1}$ are both on the second row.

As in case (5) we obtain $F_{j_{i+l}+1}$.
There is some point $\xi'_{i+l+1}$ on $C_{j_{i+l+1}}$ in $F'$ and $\xi _{i+l+1}$ on $C_{j_{i+l+1}}$ with $\xi' _{i+l+1}+n_lv_{j_{i+l+1}}$ is equivalent to $\xi_{j+l+1}+n_lw_{j_{i+l+1}}$ in case $\tilde{\xi}_{i+l} \neq w_{j_{i+l}+1}$ and $\xi'_{j+l+1}+(n_l+1)v_{j_{i+l+1}}$ is equivalent to $\xi _{j+l+1}+(n_l+1)w_{j_{i+l+1}}$ in case $\tilde{\xi}_{i+l} = w_{j_{i+l}+1}$.
In both cases we have $\xi _{i+l+1} \neq w_{j_{i+l+1}}$ (in this case we use that $s+u<s'+v$ since $t$ is a tableau).

So from $F'_{j_{i+l+1}}$ we obtain $n_lv_{j_{i+l+1}+1}$ in case $\tilde{\xi}_{i+l} \neq w_{j_{i+l}+1}$ and $(n_l+1)v_{j_{i+l+1}+1}$ in case $\tilde{\xi}_{i+l}=w_{j_{i+l}+1}$ and we have $F_{j_{i+l+1}}=F_{j_{i+l}+1}+\xi_{i+l+1}$.
In this case we have $w_{j_{i+l+1}+1}$ on $F'$ if $n_l$ is even and $v_{j_{i+l+1}+1}$ on $F'$ in case $n_l$ is odd.
Therefore from $F'_{j_{i+l+1}+1}$ we obtain $(n_l+1)v_{j_{i+l+2}}$ and $v_{j_{i+l+1}+1}$ is added to $F_{j_{i+l+1}}$ to obtain $F_{j_{i+l+1}+1}$ in case $\tilde{\xi}_{i+l}=w_{j_{i+l}+1}$.
Since $n_{l+2}=n_l+1$ also in this case we can continue the proof.
\item $j_{i+l}$ and $j_{i+l+1}$ both are on the first row.

Arguing as before from $F'_{j_{i+l+1}+1}$ we obtain $(n_l+3)v_{j_{i+l+2}}$ and $F_{j_{i+l+1}+1}=F_{j_{i+l}-1}+\xi _{i+l+1}$ with $\xi _{i+l+1} \in C_{j_{i+l}+1} \setminus \{w_{j_{i+l}+1}\}$ or $F_{j_{i+l+1}+1}=F_{j_{i+l}-1}+v_{j_{i+l+1}}$.
Otherwise $F_{j_{i+l+1}+1}=F_{j_{i+l+1}}$.
Since $n_{l+2}=n_l+3$ in this case we can continue the proof.

\item $j_{i+l}$ is on the first row and $j_{i+l+1}$ is on the second row.

Arguing as before from $F'_{j_{i+l+1}+1}$ we obtain $(n_l+2)v_{j_{i+l+2}}$ and $F_{j_{i+l+1}+1}=F_{j_{i+l}-1}+\tilde{\xi} _{i+l+1}+\xi _{i+l+1}$ with $\tilde{\xi}_{i+l+1} \in C_{j_{i+l}+1} \setminus \{w_{j_{i+l}+1}\}$ and $\xi _{i+l+1} \in C_{j_{i+l+1}} \setminus \{ w_{j_{i+l+1}} \}$ or $F_{j_{i+l+1}+1}=F_{j_{i+l}-1}+\xi _{i+l+1} + v_{j_{i+l+1}+1}$.
Since $n_{l+2}=n_l+2$ also in this case we can continue the proof.

\end{enumerate}
\end{proof}

\section{The case of discrete chains of cycles}\label{section6}

We start by recalling some definitions on finite graphs (see \cite{ref16}).

\begin{definition}\label{deffinitegraph}
A finite graph $G$ corresponds to two finite sets:
\begin{description}
\item $V(G)$, the set of vertices of $G$
\item $E(G)$, the set of edges of $G$
\end{description}
such that for each $e \in E(G)$ there is a subset $V(e) \subset V(G)$ consisting of two different elements called the vertices of the edge $e$.
\end{definition}

We assume a finite graph $G$ is connected.
This means, if $u,v \in V(G)$ with $u \neq v$, then there exists a sequence of edges $e_1, \cdots , e_t$ ($t \geq 1$) such that $u \in V(e_1)$, $v\in V(e_t)$ and $V(e_{i+1})\cap V(e_i) \neq \emptyset$ for $1 \leq e <t$.

\begin{definition}\label{divisorfingraph}
A divisor $D$ on a finite graph $G$ is an abstract combination $\sum _{v \in V(G)} D(v)v$ with $D(v) \in \mathbb{Z}$.
It is called an effective divisor if $D(v) \geq 0$ for all $v \in V(G)$.
\end{definition}

\begin{definition}\label{princdivisorfingraph}
Let $G$ be a finite graph.
Given $f : V(G) \rightarrow \mathbb{Z}$ we define a divisor $\divi (f)$ on $G$ as follows.
For $v \in V(G)$ we take
\[
\divi (f)(v) = \sum _{e \in E(G) : v \in V(e)} (f(v)-f(w_e))
\] 
with $V(e) = \{v, w_e \}$ in case $v \in V(e)$.
Those divisors are called principle divisors.

Two divisors $D_1$ and $D_2$ on $G$ are called linearly equivalent (written $D_1 \sim D_2$) if their difference $D_1 -D_2$ is a principle divisor.
\end{definition}

\begin{definition}\label{rankfingraph}
For a divisor $D$ on a finite graph $G$ one defines its rank as follows.
\begin{description}
\item $\rank (D)=-1$ in case $D$ is not linearly equivalent to an effective divisor, otherwise
\item $\rank (D)$ is the maximal integer $r \geq 0$ such that for each effective divisor $E$ of degree $r$ there exists an effective divisor $D'$ linearly equivalent to $D$ and containing $E$ (meaning $D'-E$ is an effective divisor).
\end{description}
\end{definition}

As in the case of metric graphs for a finite graph $G$ we can define the gonality $\gon (G)$ and the gonality sequence $g_r(G)$ ($r \in \mathbb{Z}_{\geq 1}$) (see Definition \ref{def7}) and the Clifford index $\Cliff (G)$ (see Definition \ref{CliffIndMetricGraph}).

\begin{definition} \label{metricgraphfingraph}
Let $G$ be a finite graph. 
The associated metric graph $\Gamma (G)$ is obtained by taking for each $e \in E(G)$ a line segment of length 1 connecting the vertices in $V(e)$.
\end{definition}

If $G$ is a finite graph and $D$ is a divisor on $G$ then $D$ is also a divisor on the metric graph $\Gamma (G)$.

\begin{lemma} \label{linequimetricfingraph}
If $D_1$ and $D_2$ are linearly equivalent divisors on a finite graph $G$ then $D_1$ and $D_2$ are also linearly equivalent as divisors on the metric graph $\Gamma (G)$.
\end{lemma}

\begin{proof}
Take $f : V(G) \rightarrow \mathbb{Z}$ such that $\divi (f) = D_1 - D_2$.
Construct a rational function $\tilde{f} : \Gamma (G) \rightarrow \mathbb{R}$ as follows.
Choose an arbitrary vertex $v_0$ of $V(G)$ and put $\tilde{f}(v_0)=0$.
For each $e \in E(G)$ with $V(e)= \{ u,v \}$, on the segment of $\Gamma (G)$ corresponding to $e$ from $u$ to $v$ we take $\tilde {f}$ affine with slope $f(v)-f(u)$.
In this way $\tilde {f}$ is completely and well-defined and $\divi (\tilde {f}) = D_1 - D_2$ on $\Gamma (G)$.
\end{proof}

For a divisor $D$ on a finite graph $G$ we also write $\rank _G(D)$ to denote the rank of $D$ as a divisor on $G$ and $\rank _{\Gamma (G)}(D)$ as a divisor on $\Gamma (G)$ (further on, when we write $\rank (D)$ it still means $\rank _G(D)$).
The following result comes from \cite {ref17}.

\begin{proposition} \label{rankfinmetricgraph}
If $G$ is a finite graph and $D$ is a divisor on $G$, then $\rank _G (D)=\rank _{\Gamma (G)}(D)$.
\end{proposition} 

We now define the discrete notion of chains of cycles and the associated profile sequence.

\begin{definition} \label{finitecycle}
A finite cycle $G$ of size $k \in \mathbb{Z}_{\geq 2}$ is a finite graph $G$ with $V(G) = \{ v_1, \cdots ,v_k \}$ and $E(G)= \{ \{ v_1,v_2 \}, \{ v_2,v_3 \}, \cdots , \{v_{k-1},v_k \}, \{ v_k,v_1 \} \}$.
So $G$ is completely determined by putting the order on $V(G)$ and using indices compatible with that order.
\end{definition}

\begin{figure}[h]
\begin{center}
\includegraphics[height=2 cm]{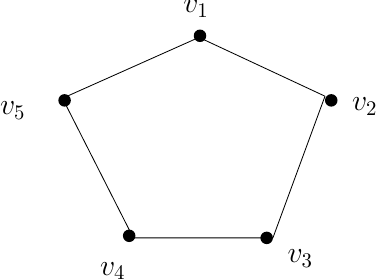}
\caption{A finite cycle of length 5 }\label{Figuur 3}
\end{center}
\end{figure}

\begin{definition} \label{discretechaincycles}
A discrete chain of cycles $G$ of genus $g$ is obtained as follows.
Take $g$ finite cycles $G_1, \cdots, G_g$ of sizes resp. $k_1, \cdots, k_g$.
For $1 \leq i \leq g$ we write $V(G_i)= \{ v_{i,1}, \cdots , v_{i,k_i} \} $ (with the order as in Definition \ref{finitecycle}).
For $1 \leq i \leq g$ choose $v_{i,j_i} \neq v_{i,1}$.
Now we put
\[
V(G)=\cup _{i=1}^g V(G_i)
\]
\[
E(G)=(\cup _{i=1}^g E(G_i)) \cup \{ \{ v_{1,j_1},v_{2,1} \}, \{ v_{2,j_2},v_{3,1} \}, \cdots , \{v_{g-1,j_{g-1}},v_{g,1} \} \}
\]
\end{definition}

\begin{figure}[h]
\begin{center}
\includegraphics[height=2 cm]{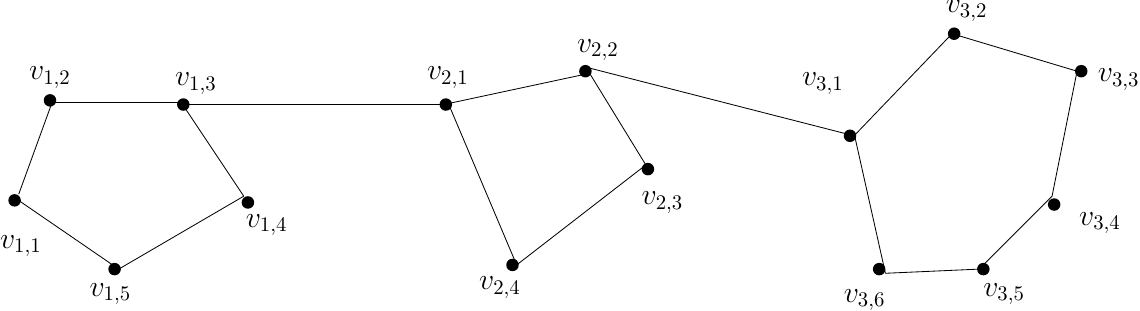}
\caption{A discrete chain of cycles of genus 3 }\label{Figuur 4}
\end{center}
\end{figure}

In what follows, we use the notations from Definition \ref{discretechaincycles}.

\begin{definition} \label{profileseqfingraph}
Associated to a discrete chain of cycles $G$ of genus $g$ there is a profile sequence $\underline{m}=(m_2, \cdots, m_g)$ defined as follows.
For $2 \leq i \leq g$ let $m_i$ be the smallest positive integer such that $m_i(j_i-1)$ is an integer multiple of $k_i$.
\end{definition}

\begin{definition} \label{notationpointdiscretechaincycles}
For $\xi \in \mathbb{Z} [ \frac{1}{j_i-1} ]$ we write $< \xi >_i$ to denote the vertex $v_{i,j}$ on $V(G_i)$ such that $(j_i -1) \xi +j_i \equiv j \mod{k_i}$.
\end{definition}

Note that $<0>_i=v_{i,j_i}$ and $<-1>_i=v_{i,1}$.

Of course values $m_i=0$ cannot occur if $G$ is a discrete chain of cycles.
Moreover $\underline{m}$ is also the profile sequence of $\Gamma (G)$.
Also for $\xi \in \mathbb{Z}[\frac {1}{j_i-1}]$ we have $<\xi>_i$ on $\Gamma (G)$ is equal to $<\xi >_i$ on $G$.

On a discrete chain of cycles we prove the analogon of Lemma \ref{lemma1}.

\begin{lemma} \label{lemma1discrete}
Let $D$ be any divisor of degree $d$ on a discrete chain of cycles $G$.
Then $D$ is linearly equivalent to a unique divisor of the form $\sum _{i=1}^g v_{i,j'_i} +(d-g)v_{g,j_g}$.
\end{lemma}

\begin{proof}
Without loss of generality we can assume $d=g$.

\underline{Proof of the existence}

For $1 \leq i \leq g-1$, taking $f_i : V(G) \rightarrow \mathbb{Z}$ with $f_i(v)=0$ if $v \in V(G_1 \cup \cdots \cup G_i)$ and $f_i(v)=1$ if $v \in V(G_{i+1} \cup \cdots \cup G_g)$, we have $\divi (f_i)=v_{i+1,1}-v_{i,j_i}$.
So the divisor $v_{i+1,1}-v_{i,j_i}$ is linearly equivalent to 0.
Adding a suited integer linear combination of such divisors to $D$, we can assume that each divisor $D_i=D \vert _{G_i}$ has degree 1.

From the Riemann-Roch Theorem (see \cite{ref16}) applied to $D_i$ on $G_i$ it follows we can assume $D_i$ as a divisor on $G_i$ is linearly equivalent to $v_{i,j'_i}$ for some $1\leq j'_i \leq k_i$.
Let $f'_i : V(G_i) \rightarrow \mathbb{Z}$ such that $\divi (f'_i)=D_i - v_{i,j'_i}$ on $G_i$.
Take $f_i :V(G) \rightarrow \mathbb{Z}$ with $f_i \vert _{V(G_i)}=f'_i$; $f_i (v) = f'_i (v_{i,j_i})$ for all $v \in V(G_{i+1} \cup \cdots \cup G_g)$ and $f_i (v) = f'_i (v_{i,1})$ for all $v \in V(G_1 \cup \cdots \cup G_{i-1})$.
Then $\divi (f_i)=D_i -v_{i,j'_i}$ on $G$.
So we conclude $D$ is linearly equivalent to $\sum _{i=1}^g v_{i,j'_i}$.

\underline{proof of the uniqueness}

Suppose $\sum _{i=1}^g v_{i,j'_i}$ and $\sum _{i=1}^g v_{i,j''_i}$ are linearly equivalent on $G$ with $1 \leq j'_i, j''_i \leq k_i$.
From Lemma \ref{linequimetricfingraph} we know that $\sum _{i=1}^g v_{i,j'_i}$ and $\sum _{i=1}^g v_{i,j''_i}$ are linearly equivalent as divisors on $\Gamma (G)$.
Then equality of those two divisors follows from the uniqueness in Lemma \ref{lemma1}.
\end{proof}

As in the metric case, for a divisor $D$ on the discrete chain of cycles $G$ we call this unique divisor $\sum _{i=1}^g v_{i,j'_i} + (d-g)v_{g,j_g}$ the representing divisor of $D$.

Now we can prove the discrete version of Theorem \ref{theorem1}.

\begin{theorem} \label{theorem1disc}
Let $D$ be a divisor of degree $d$ on a discrete chain of cycles $G$ and let $\sum _{i=1}^g v_{i,j'_i} + (d-g) v_{g,j_g}$ be the representing divisor.
Then $\rank (D) \geq r$ if and only if there exists an $\underline{m}$-displacement tableau $t$ on $[(g-d+r) \times (r+1)]$ such that for $(x,y) \in [(g-d+r) \times (r+1)]$ we have $v_{t(x,y),j'_{t(x,y)}} = <x-y>_{t(x,y)}$.
\end{theorem}

\begin{proof}
Consider the chain of cycles $\Gamma (G)$ associated to $G$.
We know from Proposition \ref{rankfinmetricgraph} that $\rank _G(D) \geq r$ is equivalent to $\rank _{\Gamma (G)}D) \geq r$.
From Theorem \ref{theorem1} we know this is equivalent to the existence of an $\underline{m}$-displacement tableau $t$ on $[(g-d+r) \times (r+1)]$ such that for the representing divisor $\sum _{i=1}^g <\xi _i>_i + (d-g)v_{g,j_g}$ of $D$ on $\Gamma (G)$ we have $<\xi _{t(x,y)} >_{t(x,y)} = <x-y>_{t(x,y)}$ for all $(x,y) \in [(g-d+r) \times (r+1)]$.
From Lemma \ref{linequimetricfingraph} we know that $\sum _{i=1}^g v_{i,j'_i} + (d-g)v_{g,j_g}$ is linearly equivalent to $D$ as a divisor on $\Gamma (G)$.
From the uniqueness in Lemma \ref{lemma1} we therefore obtain $<\xi _i>_i=v_{i,j'_i}$, so we obtain the proof of the theorem.
\end{proof}

\begin{proposition} \label{relcliffgondiscgraph}
For a discrete chain of cycles $G$ we have $\Cliff (G)=\gon (G)-2$.
\end{proposition}

\begin{proof}
Because of Theorem \ref{theorem1disc} this statement can be reduced to a statement on $\underline{m}$-diplacement tableaux (see the argument in Section 2 of \cite{ref12}).
But this statement on $\underline{m}$-displacement tableaux is proved in \cite{ref12}.
\end{proof}

\begin{definition} \label{MartensSpecialdiscrete}
A discrete chain of cycles $G$ is called a general Martens-special discrete chain of cycles of type $k$ if there exist $3 \leq j_1 \leq j_2 \leq \cdots \leq j_k \leq g-2$ with $j_{i+1}-j_i \geq 2$ for $1 \leq i \leq k-1$ such that $m_i=2$ for $i \in \{2, \cdots ,g-1 \} \setminus \{ j_1, \cdots , j_k \}$ and $m_{j_i} > g$ for $1 \leq i \leq k$.
\end{definition}

\begin{proof} [Proof of Theorem C]
In the proof we use the notations from Definition \ref{MartensSpecialdiscrete}.
On $[(g-k-1) \times 2]$ we have the $\underline{m}$-displacement tableau defined by $t(x,y)=i_{x+y-1}$ with $\{ 1, \cdots , g \} \setminus \{j_1, \cdots , j_k \} = \{ i_1 < \cdots < i_{g-k} \}$.
Because of Theorem \ref{theorem1disc} this implies on $G$ there is a divisor of degree $k+2$ and rank at least 1.
This shows $\gon (G) \leq k+2$.

The proof of Proposition \ref{prop1} shows that there is no $\underline{m}$-displacement tableau $t$ on $[(g-l-1) \times 2]$ with $l<k$ if $m_i = 2$ if $i \in \{ 2, \cdots , g-1 \} \setminus \{j_1 , \cdots , j_k \}$ and each value $j_i$ with $1 \leq i \leq k$ can occur at most once in the tableau $t$.
This condition holds because $m_{j_i}>g$ for $1 \leq i \leq k$.
So the gonality of $G$ is exactly $k+2$.

From Proposition \ref{relcliffgondiscgraph} and the Riemann-Roch Theorem for finite graphs we find using the arguments from Lemma \ref{lemmaE1} that $g_1(G)=k+2$ implies $g_r(G)=g+r$ in case $r \geq g$, $g_r(G)=g-1+r$ in case $g-g_1(G)+1 \leq r \leq g-1$ and $g_1(G)+2r-2 \leq g_r(G) \leq g-1+r$ in case $1 \leq r \leq g-g_1(G)$.
The $\underline{m}$-displacement tableaux used in the proof of Theorem B imply $g_r(G)=g_1(G)+2r-2$ in case $1 \leq r \leq g-g_1(G)$.
\end{proof}

\begin{remark}
In the case of a finite graph $G$ the invariants $\dim (W^r_d(G))$ have no meaning but the invariants $w^r_d (G)$ can be defined as in the metric case.
It is not obvious what is the relation between $w^r_d(G)$ and $w^r_d(\Gamma (G))$.
Also the proof of H. Martens' Theorem for metric graphs in \cite{ref1} can not be applied to the discrete case.
Indeed in \cite {ref1} it is crucial that there exist points $<\xi >_i$ with $\xi \in \mathbb{R}$ generally chosen, which does not exist in the discrete case.
For the same reason it is not clear whether general Martens-special discrete chains of cycles are divisorial complete.
\end{remark}

\end{document}